\documentclass{article}

\usepackage{amsmath,amssymb,amsthm}
\usepackage{a4wide}
\usepackage{subfig}
\usepackage{pdfsync}
\usepackage{graphicx}
\usepackage{url}

\newtheorem{theorem}{Theorem}[section]

\theoremstyle{remark}

\theoremstyle{definition}

\def\R{\mathbb{R}}

\renewcommand{\leq}{\leqslant}

\begin{document}

\title{An optimal control approach to malaria prevention\\
\emph{via} insecticide-treated nets\footnote{Submitted 10-May-2013;
accepted after minor revision 09-June-2013;
Conference Papers in Mathematics, Volume 2013, Article ID 658468.
http://dx.doi.org/10.1155/2013/658468}}

\author{Cristiana J. Silva\\
\texttt{cjoaosilva@ua.pt}
\and
Delfim F. M. Torres\footnote{Corresponding author.
Tel.: +351 234370668; fax: +351 234370066.}\\
\texttt{delfim@ua.pt}}

\date{CIDMA -- Center for Research and Development in Mathematics and Applications,\\
Department of Mathematics, University of Aveiro, 3810-193 Aveiro, Portugal}

\maketitle


\begin{abstract}
Malaria is a life threatening disease, entirely preventable and treatable,
provided the currently recommended interventions are properly implemented.
These interventions include vector control through the use
of insecticide-treated nets (ITNs). However, ITN possession does not necessarily
translate into use. Human behavior change interventions, including information,
education, communication (IEC) campaigns and post-distribution
hang-up campaigns are strongly recommended. In this paper we consider
a recent mathematical model for the effects of ITNs on the transmission
dynamics of malaria infection, which takes into account the human behavior.
We introduce in this model a \emph{supervision} control, representing
IEC campaigns for improving the ITN usage. We propose and solve
an optimal control problem where the aim is to minimize the number
of infectious humans while keeping the cost low. Numerical results are provided,
which show the effectiveness of the optimal control interventions.
\end{abstract}




\section{Introduction}

Malaria is a life threatening disease caused by \emph{Plasmodium} parasites
and transmitted from one individual to another by the bite of infected
female anopheline mosquitoes \cite{Agusto:varattract:2013,Teboh:etall:2013}.
In the human body, the parasites multiply in the liver, and then infect red blood cells.
Following World Health Organization (WHO) 2012 report, an estimated 3.3 billion people
were at risk of malaria in 2011, with populations living in sub-Saharan Africa
having the highest risk of acquiring malaria \cite{WHO:2012}. Malaria is an entirely
preventable and treatable disease, provided the currently recommended interventions are properly
implemented. Following WHO, these interventions include (i) vector control through
the use of insecticide-treated nets (ITNs), indoor residual spraying and,
in some specific settings, larval control, (ii) chemoprevention
for the most vulnerable populations, particularly pregnant
women and infants, (iii) confirmation of malaria diagnosis
through microscopy or rapid diagnostic tests for every
suspected case, and (iv) timely treatment with appropriate antimalarial
medicines \cite{WHO:2012}. An ITN is a mosquito net that repels,
disables and/or kills mosquitoes coming into contact
with insecticide on the netting material. ITNs are considered
one of the most effective interventions against malaria \cite{Lengeler:1996}.
In 2007, WHO recommended full ITN coverage of all people at risk of malaria,
even in high-transmission settings \cite{WHO:2007}.
By 2011, 32 countries in the African region and 78 other countries worldwide,
had adopted the WHO recommendation. A total of 89 countries, including 39 in Africa,
distribute ITNs free of charge. Between 2004 and 2010, the number of ITNs delivered
annually by manufacturers to malaria-endemic countries in sub-Saharan Africa
increased from 6 million to 145 million. However, the numbers delivered in 2011 and 2012
are below the number of ITNs required to protect all population at risk.
There is an urgent need to identify new funding sources
to maintain and expand coverage levels of interventions so that
outbreaks of disease can be avoided and international targets
for reducing malaria cases and deaths can be attained \cite{WHO:2012}.

A number of studies reported that ITN possession does not necessarily translate into use.
Human behavior change interventions, including information, education, communication (IEC)
campaigns and post-distribution hang-up campaigns are strongly recommended, especially
where there is evidence of their effectiveness in improving ITN usage
\cite{Afolabi_et_all:MJ:2009,Macintyre:2006,WHO:2012}.
In this paper we consider the model from \cite{Agusto_et_all:JTB:2013}
for the effects of ITNs on the transmission dynamics of malaria infection.
Other articles considered the impact of intervention strategies using ITN
(see, e.g., \cite{Killeen:2007,Smith:et:all:Paras:2008}). However,
only in \cite{Agusto_et_all:JTB:2013} the human behavior is incorporated into the model.
We introduce in the model of \cite{Agusto_et_all:JTB:2013} a \emph{supervision} control,
$u$, which represents IEC campaigns for improving the ITN usage. The reader interested
in the use of optimal control to infectious diseases is referred
to \cite{Rodrigues:Torres:2010,Silva:Torres:NACO2012} and references cited therein.
For the state of art in malaria research see \cite{Teboh:etall:2013}.

The text is organized as follows.
In Section~\ref{sec:cont:model} we present the mathematical model
for malaria transmission with one control function $u$. In Section~\ref{sec:ocp}
we propose an optimal control problem for the minimization of the number
of infectious humans while controlling the cost of control interventions.
Finally, in Section~\ref{sec:num:simu} some numerical results
are analyzed and interpreted from the epidemiological point of view.


\section{Controlled model}
\label{sec:cont:model}

We consider a mathematical model presented in \cite{Agusto_et_all:JTB:2013}
for the effects of ITN on the transmission of malaria infection and introduce
a time-dependent \emph{supervision} control $u$. The model considers transmission
of malaria infection of mosquito (also referred as vector)
and human (also referred as host) population. The host population is divided
into two compartments, susceptible ($S_h$) and infectious ($I_h$),
with a total population ($N_h$) given by $N_h = S_h + I_h$. Analogously,
the vector population is divided into two compartments, susceptible ($S_v$)
and infectious ($I_v$), with a total population ($N_v$) given by $N_v = S_v + I_v$.
The model is constructed under the following assumptions: all newborns individuals
are assumed to be susceptible and no infected individuals are assumed to come
from outside the community. The human and mosquito recruitment rates are denoted by
$\Lambda_h$ and $\Lambda_v$, respectively. The disease is fast progressing,
thus the exposed stage is minimal and is not considered. Infectious individuals
can die from the disease or become susceptible after recovery while
the mosquito population does not recover from infection. ITNs contribute
for the mortality of mosquitoes. The average number of bites per mosquito,
per unit of time (mosquito-human contact rate), is given by
\begin{equation*}
\beta = \beta_{max}(1 - b) \, ,
\end{equation*}
where $\beta_{max}$ denotes the maximum transmission rate and $b$
the proportion of ITN usage. It is assumed that the minimum transmission rate is zero.
The value of $\beta$ is the same for human and mosquito population,
so the average number of bites per human per unit of time is $\beta N_v/N_h$
(see \cite{Agusto_et_all:JTB:2013} and the references cited therein). Thus,
the force of infection for susceptible humans ($\lambda_h$)
and susceptible vectors ($\lambda_v$) are given by
\begin{equation*}
\lambda_h = \frac{p_1 \beta I_v}{N_h}
\quad \text{and} \quad \lambda_v = \frac{p_2 \beta I_h}{N_h} \, ,
\end{equation*}
where $p_1$ and $p_2$ are the transmission probability per bite from
infectious mosquitoes to humans, and from infectious humans to mosquitoes,
respectively. The death rate of the mosquitoes is modeled by
$\mu_{vb} = \mu_{v1}+ \mu_{max}b$, where $\mu_{v1}$ is the natural death rate
and $\mu_{max}b$ is the death rate due to pesticide on ITNs. The coefficient
$1-u$ represents the effort of susceptible humans that become infected
by infectious mosquitoes bites, such as educational programs/campaigns
for the correct use of ITNs, supervision teams that visit every house
in a certain region and assure that every person has access to an ITN,
know how to use it correctly, and recognize its importance on the reduction
of malaria disease transmission. The values of the parameters $\Lambda_h$,
$\Lambda_v$, $\mu_h$, $\delta_h$, $\gamma_h$, $\mu_{v1}$, $\mu_{max}b$,
$\beta_{max}$, $p_1$ and $p_2$ are taken from \cite{Agusto_et_all:JTB:2013}
(see Table~\ref{table:parameters}).

The state system of the controlled malaria model is given by
\begin{equation}
\label{model:malaria:controls}
\begin{cases}
\dot{S}_h(t) = \Lambda_h - (1-u(t)) \lambda_h  S_h(t)  + \gamma_h I_h(t) - \mu_h S_h(t) \, , \\[0.2 cm]
\dot{I}_h(t) = (1-u(t)) \lambda_h S_h(t) - (\mu_h + \gamma_h + \delta_h)I_h(t) \, , \\[0.2 cm]
\dot{S}_v(t) = \Lambda_v - \lambda_v S_v(t) - \mu_{vb} S_v(t) \, , \\[0.2 cm]
\dot{I}_v(t) = p_2 \lambda_v S_v(t) - \mu_{vb} I_v(t) \, .
\end{cases}
\end{equation}

\begin{table}[!htb]
\centering
\begin{tabular}{|l | l | l |}
\hline
{\scriptsize{Symbol}} & {\scriptsize{Description}}  & {\scriptsize{Value}} \\
\hline
{\scriptsize{$\Lambda_h$}} & {\scriptsize{Recruitment rate in humans}}
& {\scriptsize{$10^3/(70\times 365)$ }}\\
{\scriptsize{$\Lambda_v$}} & {\scriptsize{Recruitment rate in mosquitoes}}
& {\scriptsize{$10^4/21$}}\\
{\scriptsize{$\mu_h$}} & {\scriptsize{Natural mortality rate in humans}}
& {\scriptsize{$1/(70 \times 365)$}}\\
{\scriptsize{$\delta_h$}} & {\scriptsize{Disease induced mortality rate in humans }}
& {\scriptsize{$10^{-3}$}}\\
{\scriptsize{$b$}} & {\scriptsize{Proportion of treated net usage}}
& {\scriptsize{0.25; 0.3; 0.4; 0.5; 0.6; 0.7; 0.75}}\\
{\scriptsize{$\gamma_h$}} & {\scriptsize{Recovery rate of infectious humans to be susceptible}}
&{\scriptsize{$1/4$}}\\
{\scriptsize{$\mu_{v_1}$}} & {\scriptsize{Natural mortality rate of mosquitoes}}
& {\scriptsize{$1/21$}} \\
{\scriptsize{$\mu_{max}b$}} & {\scriptsize{Mortality rate of mosquitoes due to treated net}}
&  {\scriptsize{1/21}}\\
{\scriptsize{$\beta_{max}$}} & {\scriptsize{Maximum mosquito-human contact rate}}
&  {\scriptsize{$0.1$}}\\
{\scriptsize{$p_1$}} & {\scriptsize{Probability of disease transmission from mosquito}} & {\scriptsize{1}} \\
{\scriptsize{$p_2$}} & {\scriptsize{Probability of disease transmission from human to mosquito}} & {\scriptsize{$1$}}\\
{\scriptsize{$A_1$}} & {\scriptsize{Weight constant on infectious humans}} & {\scriptsize{$25$}}\\
{\scriptsize{$C$}} & {\scriptsize{Weight constant on control}} & {\scriptsize{$50$}}\\
{\scriptsize{$S_h(0)$}} & {\scriptsize{Susceptible individuals initial value}} & {\scriptsize{$800$}}\\
{\scriptsize{$I_h(0)$}} & {\scriptsize{Infectious individuals initial value}} & {\scriptsize{$200$}}\\
{\scriptsize{$S_v(0)$}} & {\scriptsize{Susceptible vectors initial value}} & {\scriptsize{$4000$}}\\
{\scriptsize{$I_v(0)$}} & {\scriptsize{Infectious vectors initial value}} & {\scriptsize{$900$}}\\
\hline
\end{tabular}
\caption{Parameter values.}
\label{table:parameters}
\end{table}

The rate of change of the total human and mosquito populations is given by
\begin{equation*}
\begin{split}
\dot{N}_h(t) &= \Lambda_h - \mu_h N_h(t) - \delta_h I_h(t) \, ,\\
\dot{N}_v(t) &= \Lambda_v - \mu_{vb} N_v(t)  \, .
\end{split}
\end{equation*}


\section{Optimal control problem}
\label{sec:ocp}

We formulate an optimal control problem that describes the goal and restrictions of the epidemic.
In \cite{Agusto_et_all:JTB:2013} it is found that
the ITN usage must attain 75\% ($b=0.75$) of the host population
in order to extinct malaria. Therefore, educational campaigns must continue
encouraging the population to use ITNs. Moreover, it is very important
to assure that ITNs are in good conditions
and each individual knows how to use them properly.
Having this in mind, we introduce a \emph{supervision} control function, $u$,
where the coefficient $1-u$ represents the effort to reduce
the number of susceptible humans that become infected by infectious
mosquitoes bites, assuring that ITNs are correctly used
by the fraction $b$ of the host population.

We consider the state system \eqref{model:malaria:controls}
of ordinary differential equations in $\R^4$
with the set of admissible control functions given by
\begin{equation*}
\Omega = \left\{ u(\cdot) \in L^\infty(0, t_f) \, | \,
0 \leq u(t) \leq 1, \, \, \forall t \in [0, t_f] \right\} \, .
\end{equation*}
The objective functional is given by
\begin{equation}
\label{cost:function:malaria:J1}
J_1(u) = \int_0^{t_f}  A_1 I_h(t) + \frac{C}{2} u^2(t)  \, dt \, ,
\end{equation}
where the weight coefficient, $C$, is a measure of the relative cost
of the interventions associated to the control $u$ and $A_1$ is the
weight coefficient for the class $I_h$. The aim is to minimize
the infectious humans while keeping the cost low. More precisely,
we propose the optimal control problem of determining
$(S_h^*, I_h^*, S_v^*, I_v^*)$ associated to an admissible control
$u^*(\cdot) \in \Omega$ on the time interval $[0, t_f]$, satisfying
\eqref{model:malaria:controls}, the initial conditions $S_h(0)$, $I_h(0)$,
$S_v(0)$ and $I_v(0)$ (see Table~\ref{table:parameters}) and minimizing
the cost function \eqref{cost:function:malaria:J1}, i.e.,
\begin{equation}
\label{min:cost:function:malaria}
J_1(u^*(\cdot)) = \min_{\Omega} J_1(u(\cdot)) \, .
\end{equation}
The existence of an optimal control $u^*(\cdot)$ comes from the convexity
of the Lagrangian of \eqref{cost:function:malaria:J1} with respect to the control
and the regularity of the system \eqref{model:malaria:controls} (see, e.g.,
\cite{Cesari_1983,Fleming_Rishel_1975} for existence results of optimal solutions).
Applying the Pontryagin maximum principle \cite{Pontryagin_et_all_1962}
we derive the optimal solution $(u^*, S_h^*, I_h^*, S_v^*, I_v^*)$
of the proposed optimal control problem (see the Appendix).

More generally, one could take the following cost function:
\begin{equation*}
J_2(u) = \int_0^{t_f} A_1 I_h(t) + A_2 I_v(t) + \frac{C}{2} u^2(t) \, dt ,
\end{equation*}
where $A_2$ is the weight constant on infectious mosquitoes (for numerical simulations we considered $A_2 = 25$).
It turns out that when we include in the objective function the number of infectious mosquitoes,
the distribution of the total host population $N_h$ and vector population $N_v$ by the categories
$S_h$, $I_h$ and $S_v$, $I_v$, respectively, is the same for both cost functions $J_1$ and $J_2$
(see Figures~\ref{fig:Sh:Ih:J1:J2} and \ref{fig:Sv:Iv:J1:J2}).
On the other hand, the effort on the control is higher for the cost function $J_2$
(see Figure~\ref{control:J1:J2}). Therefore, we choose to use the cost function $J_1$
in our numerical simulations (Section~\ref{sec:num:simu}).

\begin{figure}[!htb]
\centering
\subfloat[\footnotesize{Susceptible humans for $J_1$ and $J_2$.}]{\label{Sh:J1:J2}
\includegraphics[width=0.50\textwidth]{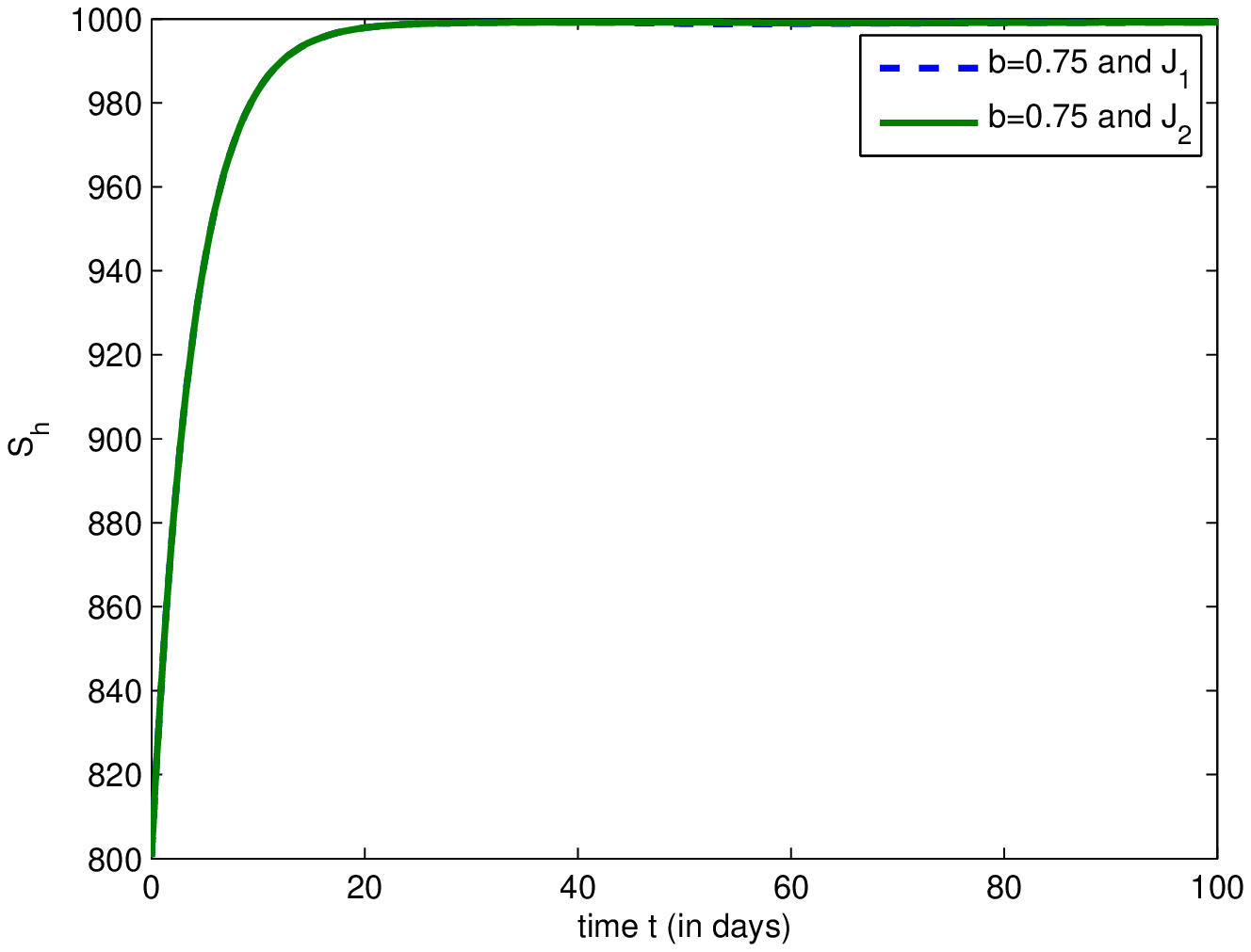}}
\subfloat[\footnotesize{Infectious humans for $J_1$ and $J_2$.}]{\label{Ih:J1:J2}
\includegraphics[width=0.50\textwidth]{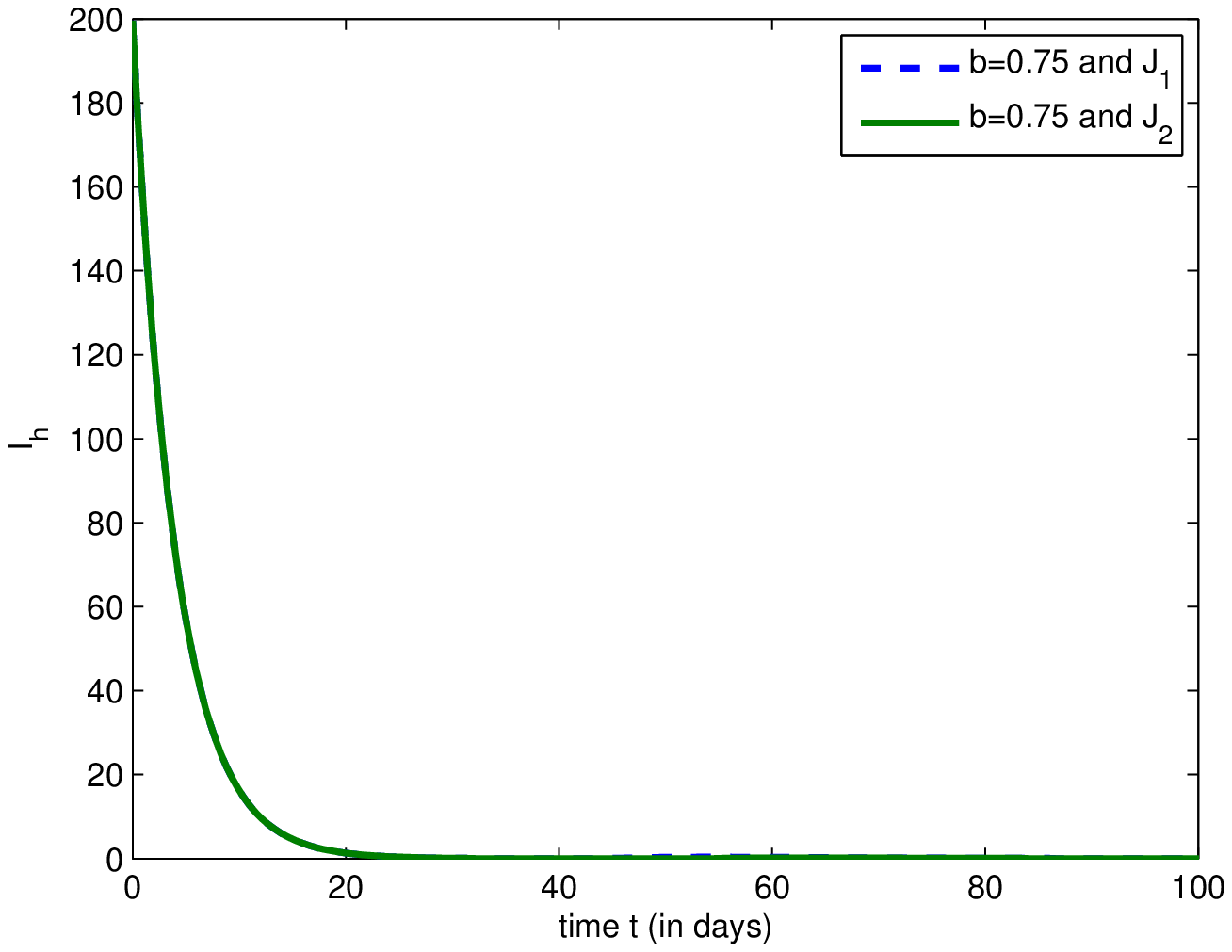}}
\caption{Susceptible and infectious individuals for different cost functions
$J_1$ and $J_2$ (parameter/constant values from Table~\ref{table:parameters} and $b=0.75$).}
\label{fig:Sh:Ih:J1:J2}
\end{figure}

\begin{figure}[!htb]
\centering
\subfloat[\footnotesize{Susceptible mosquitoes for $J_1$ and $J_2$.}]{\label{Sv:J1:J2}
\includegraphics[width=0.50\textwidth]{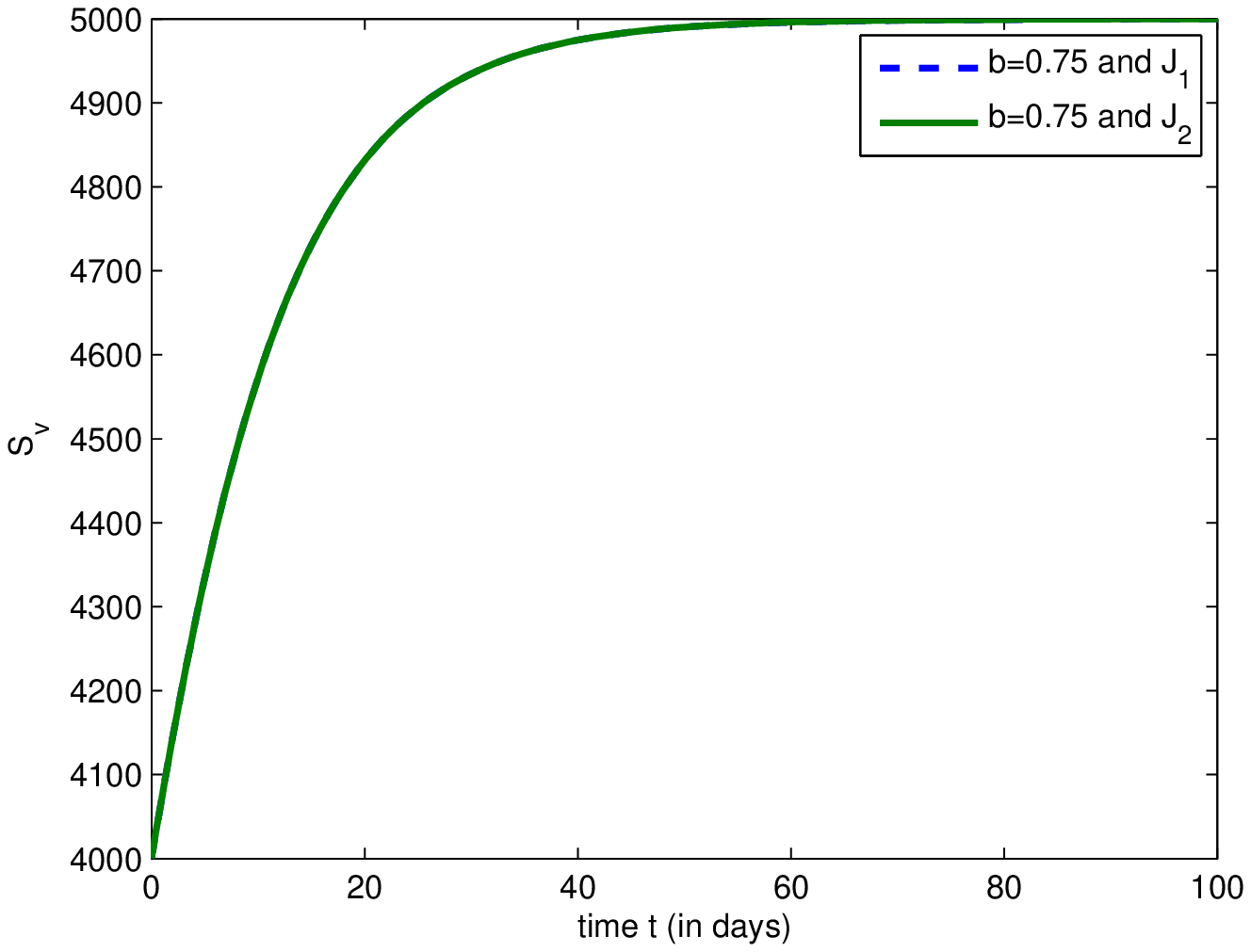}}
\subfloat[\footnotesize{Infectious mosquitoes for $J_1$ and $J_2$.}]{\label{Iv:J1:J2}
\includegraphics[width=0.50\textwidth]{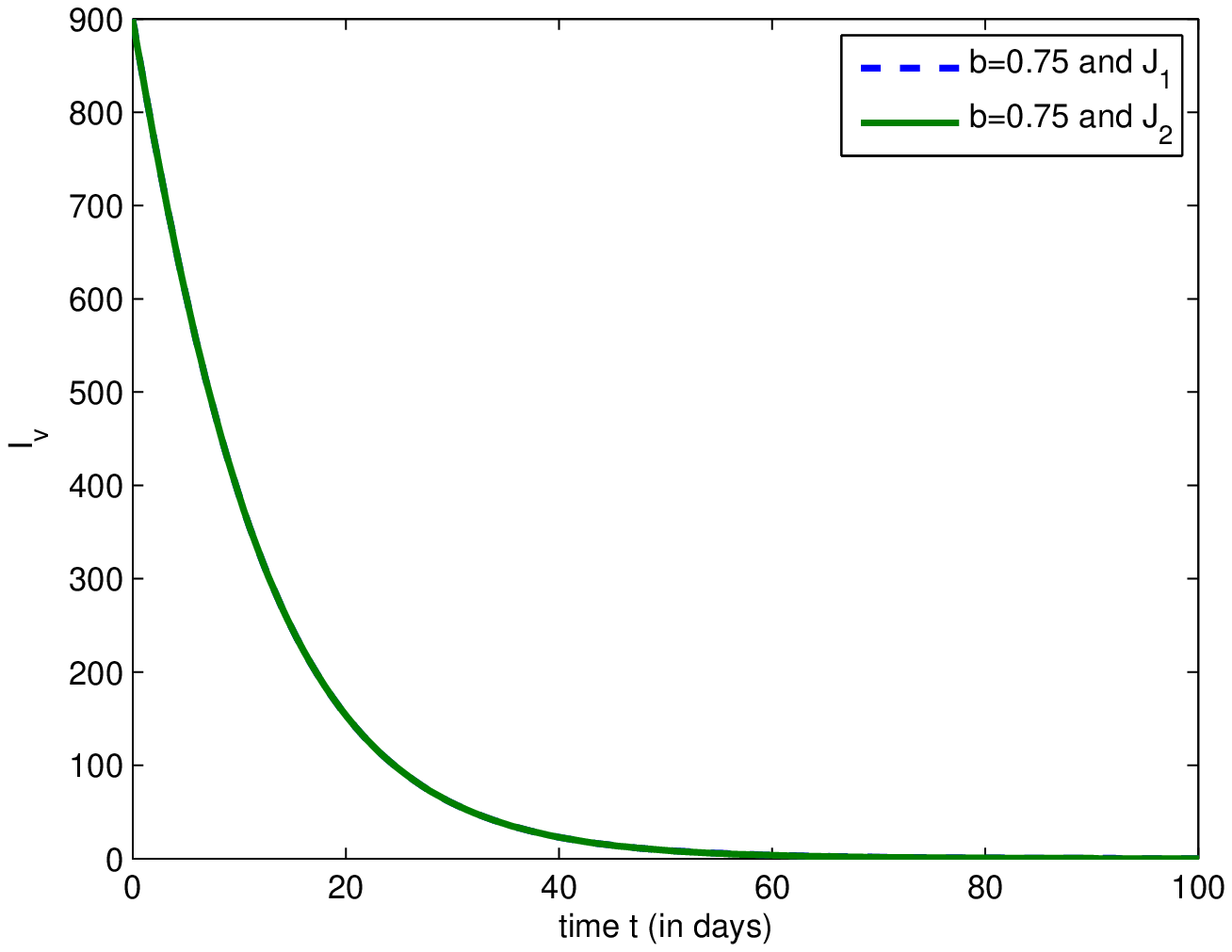}}
\caption{Susceptible and infectious mosquitoes for different cost functions $J_1$ and $J_2$
(parameter/constant values from Table~\ref{table:parameters} and $b=0.75$).}
\label{fig:Sv:Iv:J1:J2}
\end{figure}

\begin{figure}[!htb]
\centering
\includegraphics[width=0.5\textwidth]{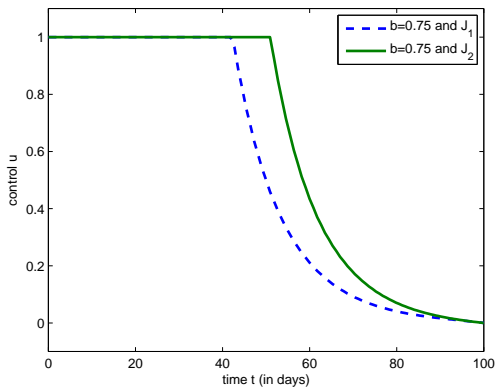}
\caption{Optimal control $u$ for different cost functions $J_1$ and $J_2$
(parameter/constant values from Table~\ref{table:parameters} and $b=0.75$).}
\label{control:J1:J2}
\end{figure}


\section{Numerical results and discussion}
\label{sec:num:simu}

Our numerical results were obtained and confirmed following
different approaches. The first approach consisted in using IPOPT \cite{IPOPT}
and the algebraic modeling language AMPL \cite{AMPL}.
In a second approach we used the PROPT
Matlab Optimal Control Software \cite{PROPT}.
The results were coincident and are easily confirmed
by the ones obtained using an iterative method
that consists in solving the system of eight ODEs given by
\eqref{model:malaria:controls} and \eqref{adjoint_function} in Appendix.
For that, one first solves system \eqref{model:malaria:controls}
with a guess for the control over the time interval
$[0, t_f]$ using a forward fourth-order Runge--Kutta scheme
and the transversality conditions \eqref{eq:trans:cond} in Appendix.
Then, system \eqref{adjoint_function} is solved by
a backward fourth-order Runge--Kutta scheme using the current
iteration solution of \eqref{model:malaria:controls}.
The controls are updated by using a convex combination of the previous controls
and the values from \eqref{optcontrols} (see Appendix).
The iterative method ends when the values of the approximations
at the previous iteration are close to the ones at the present iteration.
For details see \cite{SLenhart_2002,MyID:271}.

First of all we consider $b=0.75$ and show that when we apply the \emph{supervision}
control $u$, better results are obtained, that is, the number of infected humans
vanishes faster when compared to the case where no controls are used.
If the control intervention $u$ is applied, then the number of infectious individuals
vanishes after approximately 30 days. If no control is considered, then it takes
approximately 70 days to assure that there are no infectious humans
(see Figure~\ref{fig:Sh:Ih:b075} for the fraction of susceptible
and infectious humans and Figure~\ref{control:b075} for the optimal control).

\begin{figure}[!htb]
\centering
\subfloat[\footnotesize{Susceptible humans}]{\label{Sh:b075}
\includegraphics[width=0.50\textwidth]{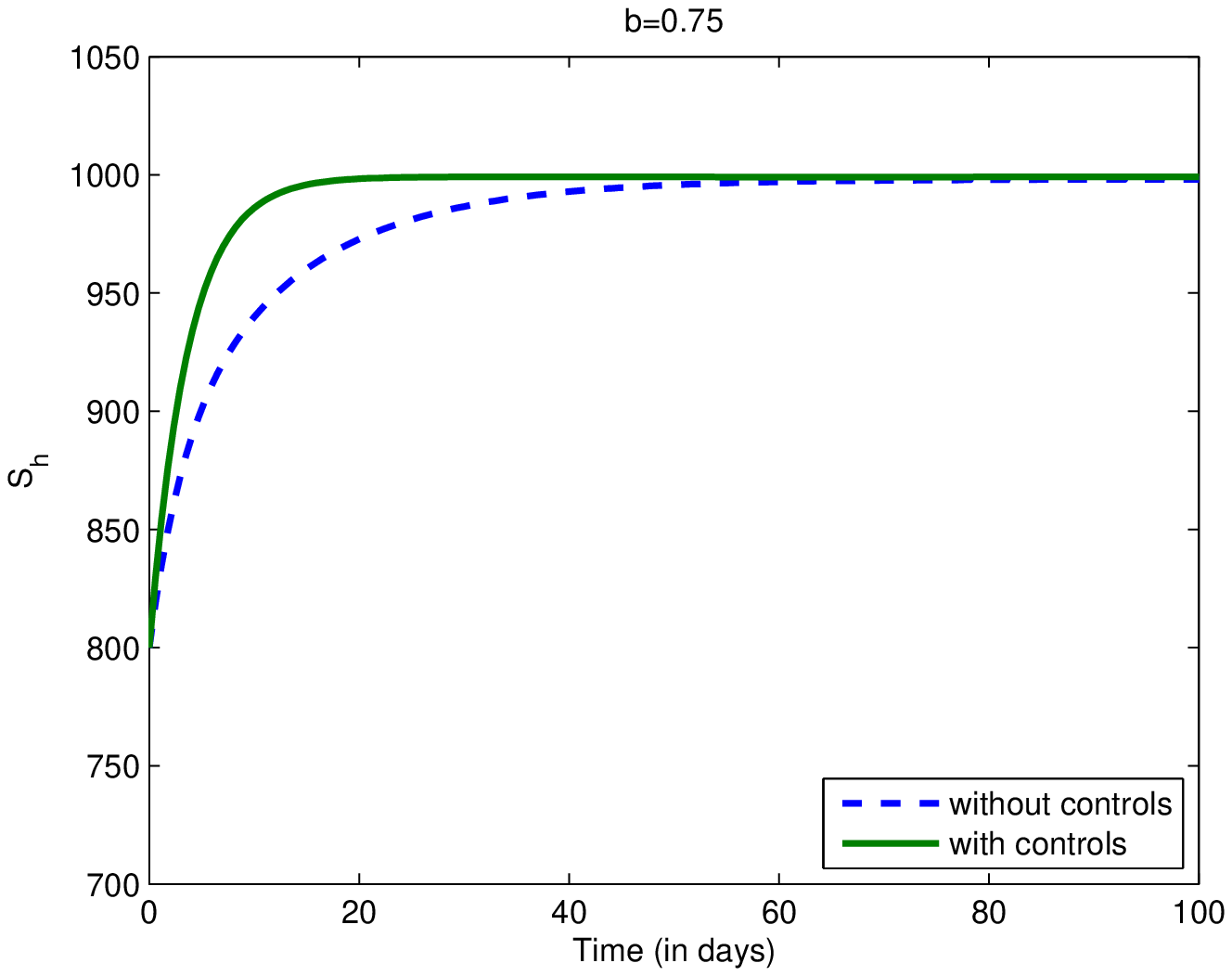}}
\subfloat[\footnotesize{Infectious humans}]{\label{Ih:075}
\includegraphics[width=0.50\textwidth]{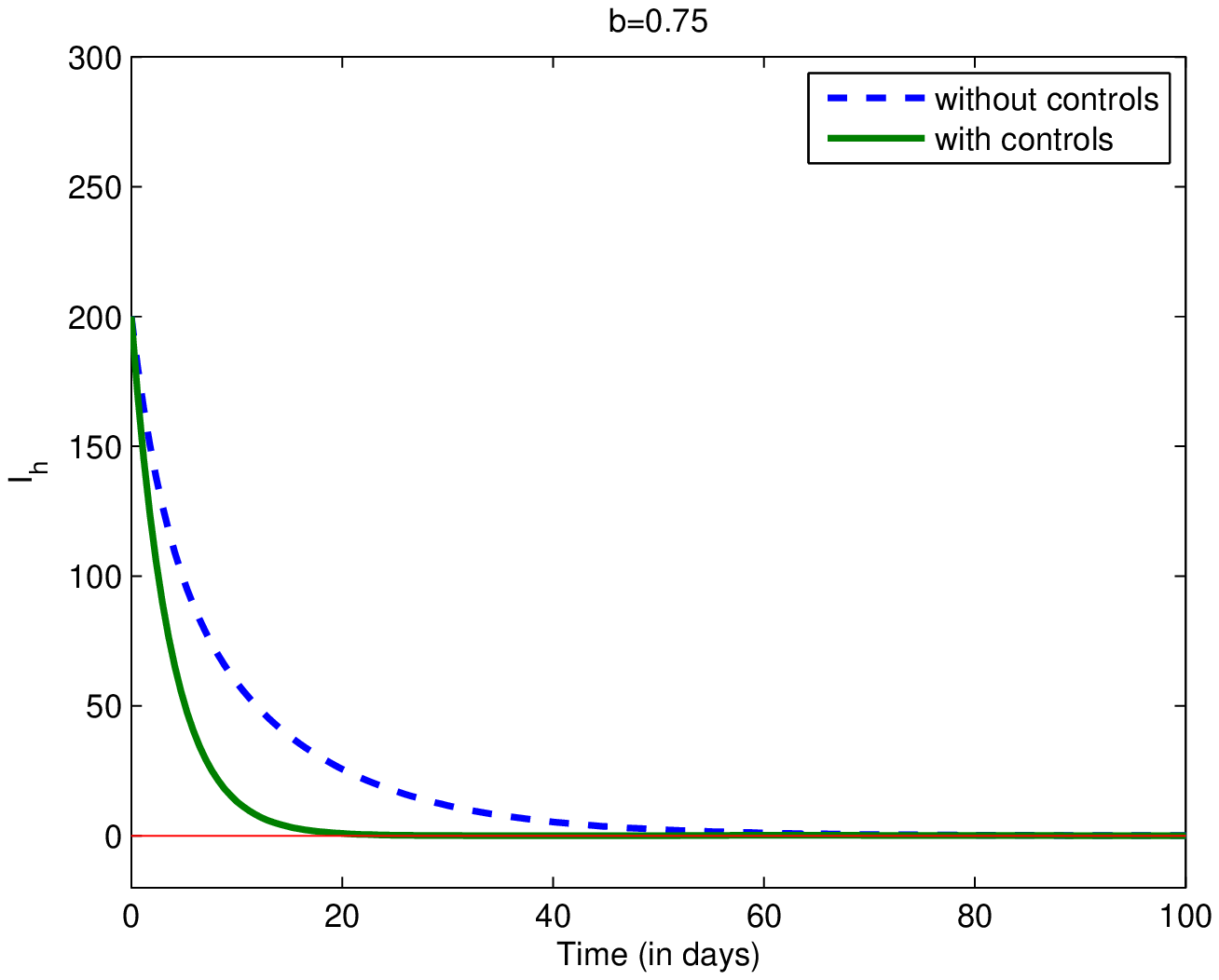}}
\caption{Susceptible and infectious humans for $b=0.75$ with and without control.}
\label{fig:Sh:Ih:b075}
\end{figure}

\begin{figure}[!htb]
\centering
\includegraphics[width=0.5\textwidth]{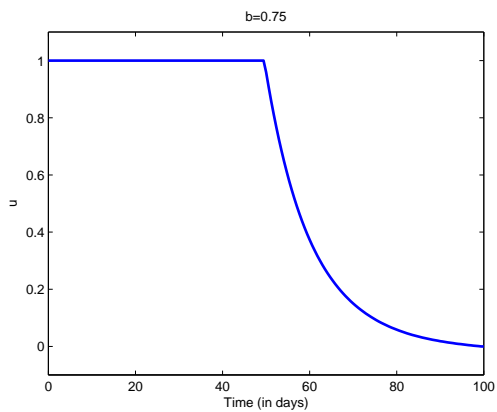}
\caption{Optimal control $u$ for $b=0.75$.}
\label{control:b075}
\end{figure}

For smaller proportions of ITN usage than $b=0.75$,
similar results on the reduction of infectious humans are attained
when we consider the optimal \emph{supervision} control $u$
(see Figures~\ref{fig:Sh:Ih:bvariable} and \ref{control:bvariable}).
We note that the control $u$ does not contribute significantly
for the decrease of $I_v$ (see Figure~\ref{fig:Sv:Iv:b075}).

\begin{figure}[!htb]
\centering
\subfloat[\footnotesize{Susceptible humans}]{\label{Sh:bvariable}
\includegraphics[width=0.50\textwidth]{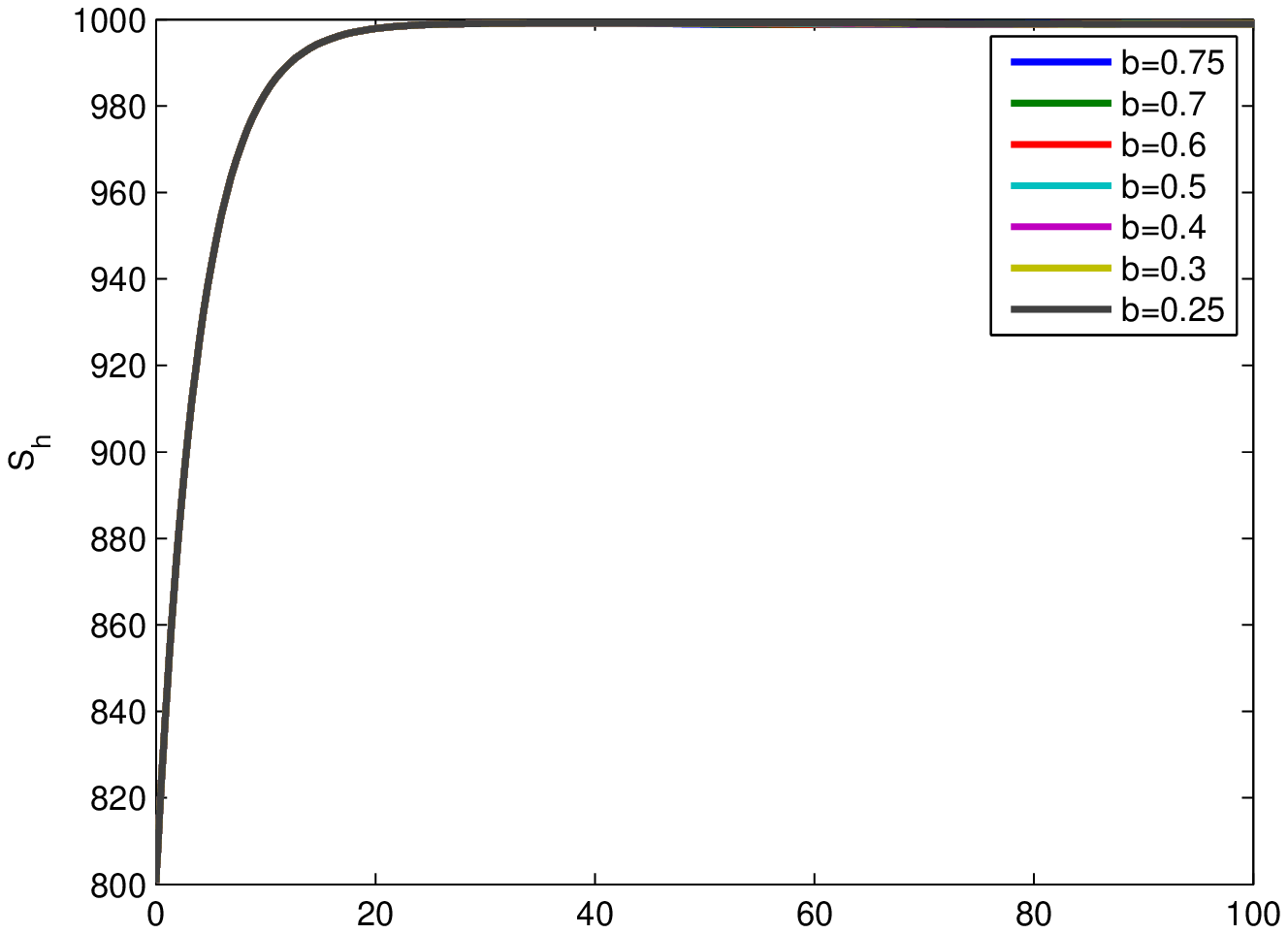}}
\subfloat[\footnotesize{Infectious humans}]{\label{Ih:bvariable}
\includegraphics[width=0.50\textwidth]{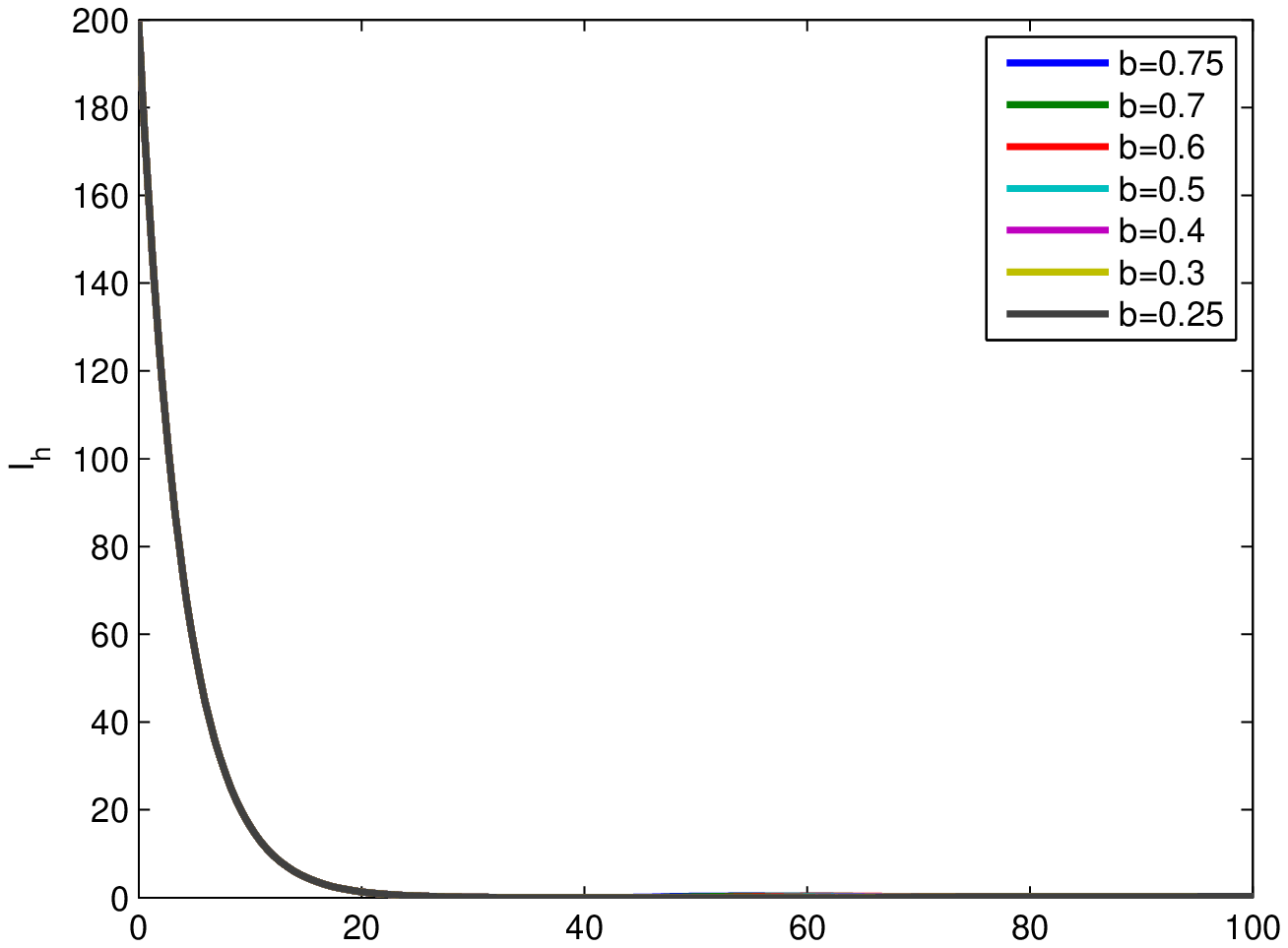}}
\caption{Susceptible and infectious humans for $b=0.25; 0.3; 0.4; 0.5; 0.6; 0.7; 0.75$.}
\label{fig:Sh:Ih:bvariable}
\end{figure}

\begin{figure}[!htb]
\centering
\includegraphics[width=0.5\textwidth]{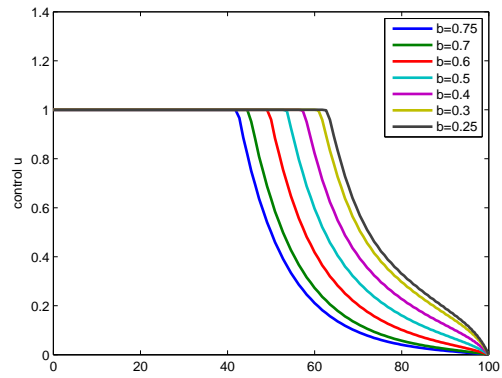}
\caption{Optimal control $u$ for $b=0.25; 0.3; 0.4; 0.5; 0.6; 0.7; 0.75$.}
\label{control:bvariable}
\end{figure}

\begin{figure}[!htb]
\centering
\subfloat[\footnotesize{Susceptible mosquitoes}]{\label{Sv:b075}
\includegraphics[width=0.50\textwidth]{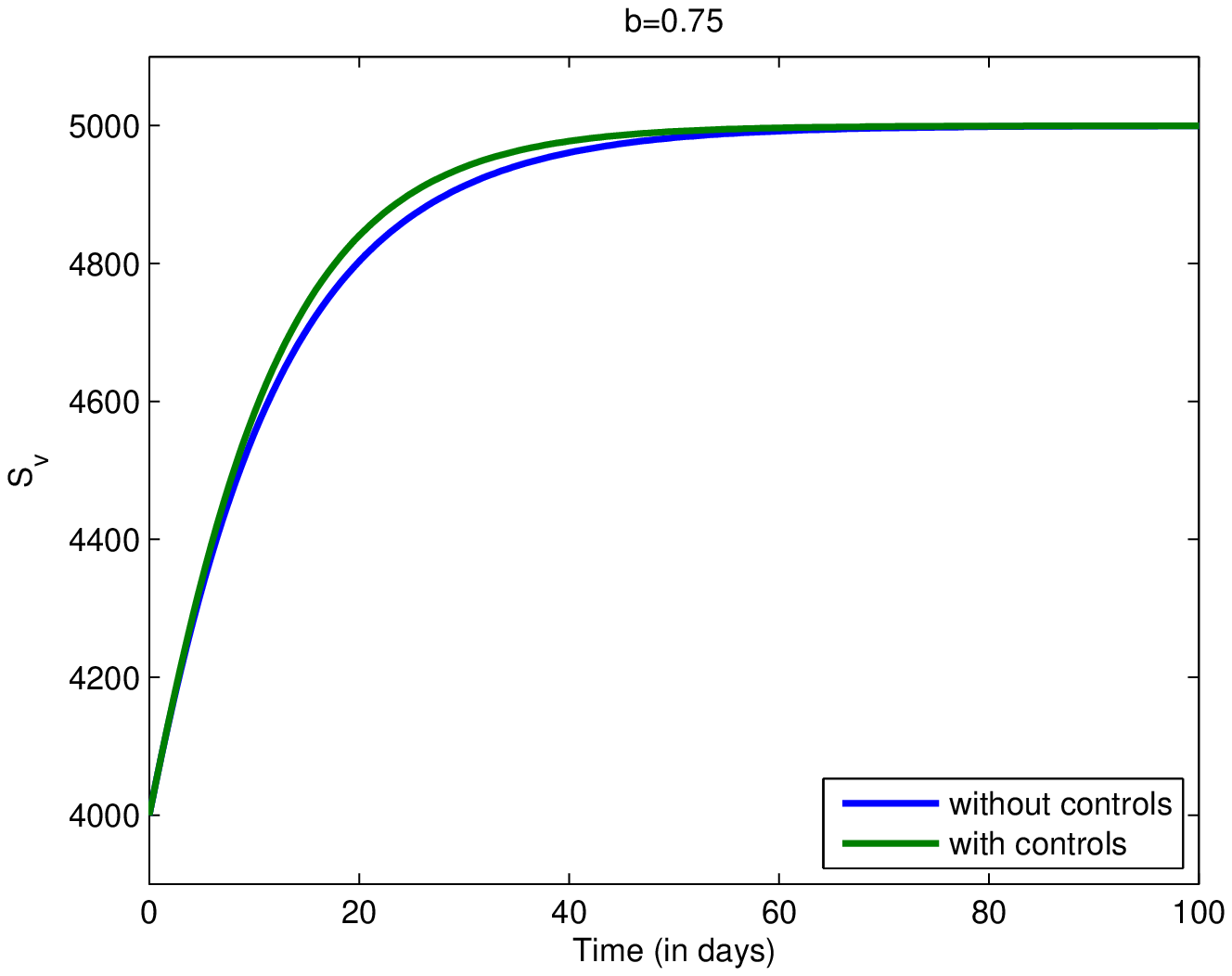}}
\subfloat[\footnotesize{Infectious mosquitoes}]{\label{Iv:075}
\includegraphics[width=0.50\textwidth]{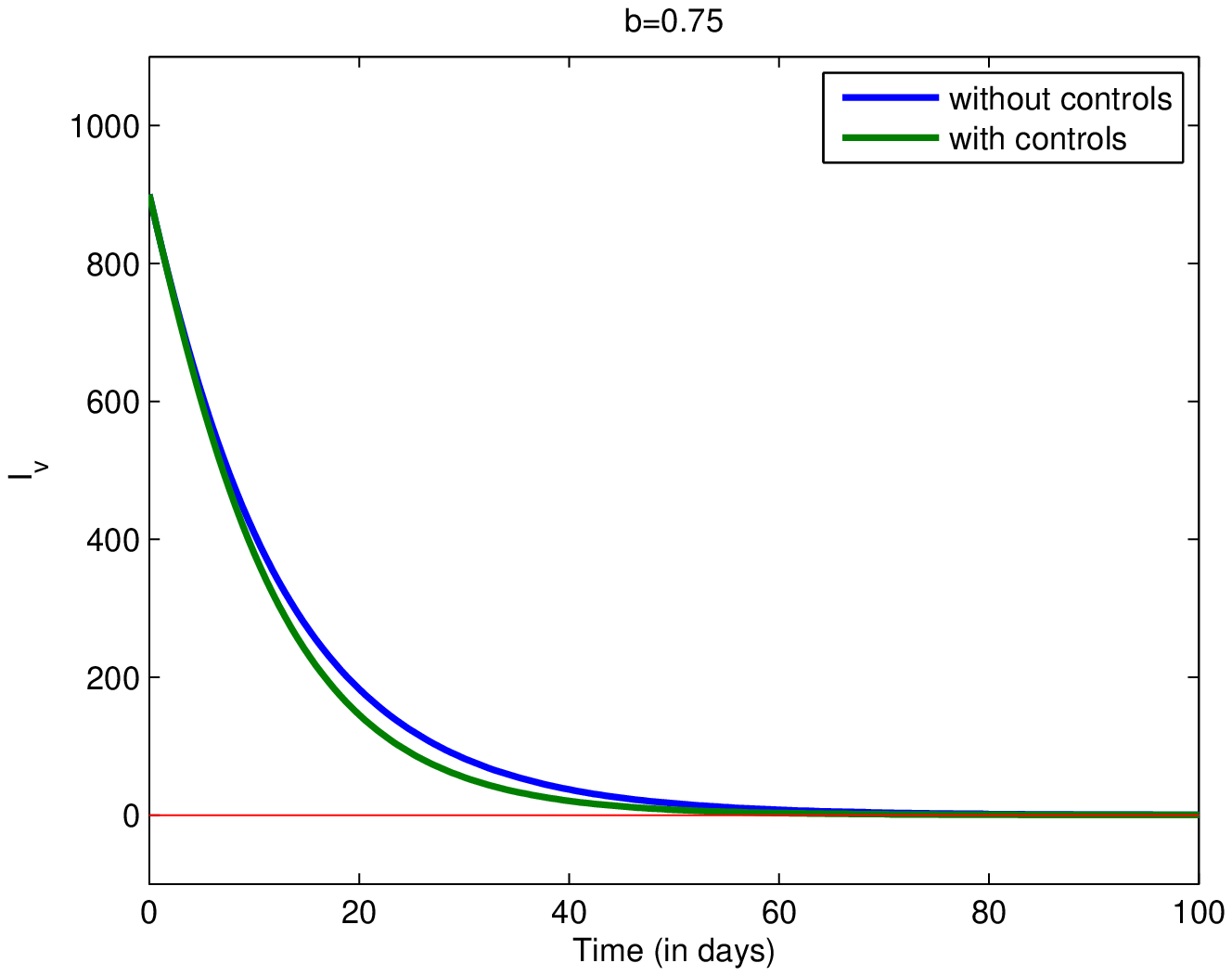}}
\caption{Susceptible and infectious mosquitoes for $b=0.75$ with and without control.}
\label{fig:Sv:Iv:b075}
\end{figure}


\section*{Appendix}

According to the Pontryagin Maximum Principle \cite{Pontryagin_et_all_1962},
if $u^*(\cdot) \in \Omega$ is optimal for the problem \eqref{model:malaria:controls},
\eqref{min:cost:function:malaria} with the initial conditions given
in Table~\ref{table:parameters} and fixed final time $t_f$, then there exists
a nontrivial absolutely continuous mapping $\lambda : [0, t_f] \to \mathbb{R}^4$,
$\lambda(t) = \left(\lambda_1(t), \lambda_2(t),
\lambda_3(t), \lambda_4(t)\right)$,
called \emph{adjoint vector}, such that
\begin{equation*}
\dot{S}_h = \frac{\partial H}{\partial \lambda_1} \, , \quad
\dot{I}_h= \frac{\partial H}{\partial \lambda_2} \, , \quad
\dot{S}_v= \frac{\partial H}{\partial \lambda_3} \, , \quad
\dot{I}_v = \frac{\partial H}{\partial \lambda_4}
\end{equation*}
and
\begin{equation}
\label{adjsystemPMP}
\dot{\lambda}_1 = -\frac{\partial H}{\partial S_h} \, , \quad
\dot{\lambda}_2 = -\frac{\partial H}{\partial I_h} \, , \quad
\dot{\lambda}_3 = -\frac{\partial H}{\partial S_v} \, , \quad
\dot{\lambda}_4 = -\frac{\partial H}{\partial I_v} \, ,
\end{equation}
where function $H$ defined by
\begin{equation*}
\begin{split}
H&= H(S_h, I_h, S_v, I_v, \lambda, u) \\
&=A_1 I_h + \frac{C}{2}u^2 \\
&\, \, + \lambda_1 \left(\Lambda_h - (1-u) \lambda_h  S_h  + \gamma_h I_h - \mu_h S_h \right)\\
&\, \, + \lambda_2 \left( (1-u) \lambda_h S_h - (\mu_h + \gamma_h + \delta_h)I_h\right)\\
&\, \, + \lambda_3 \left(\Lambda_v - \lambda_v S_v - \mu_{vb} S_v \right)\\
&\, \, + \lambda_4 \left(p_2 \lambda_v S_v - \mu_{vb} I_v \right)
\end{split}
\end{equation*}
is called the \emph{Hamiltonian}, and the minimization condition
\begin{equation}
\label{maxcondPMP}
\begin{split}
H(S_h^*(t), &I_h^*(t), S_v^*(t), I_v^*(t),
\lambda^*(t), u^*(t))\\
&= \min_{0 \leq u \leq 1}
H(S_h^*(t), I_h^*(t), S_v^*(t), I_v^*(t), \lambda^*(t), u)
\end{split}
\end{equation}
holds almost everywhere on $[0, t_f]$. Moreover, the transversality conditions
\begin{equation}
\label{eq:trans:cond}
\lambda_i(t_f) = 0, \quad
i =1,\ldots, 4 \, ,
\end{equation}
hold.

\begin{theorem}
\label{the:thm}
Problem \eqref{model:malaria:controls}, \eqref{min:cost:function:malaria}
with fixed initial conditions $S_h(0)$, $I_h(0)$, $S_v(0)$ and $I_v(0)$
and fixed final time $t_f$, admits an unique optimal solution
$\left(S_h^*(\cdot), I_h^*(\cdot), S_v^*(\cdot), I_v^*(\cdot)\right)$
associated to an optimal control $u^*(\cdot)$ on $[0, t_f]$.
Moreover, there exists adjoint functions $\lambda_1^*(\cdot)$, $\lambda_2^*(\cdot)$,
$\lambda_3^*(\cdot)$ and $\lambda_4^*(\cdot)$ such that
\begin{equation}
\label{adjoint_function}
\begin{cases}
\dot{\lambda^*_1}(t) =  \lambda^*_1(t) \left( (1-u^*(t))\lambda_h
+ \mu_h \right) - \lambda_2^*(t) \lambda_h(1-u^*(t)) \\[0.1 cm]
\dot{\lambda^*_2}(t) = - A_1 - \lambda_1^*(t) \gamma_h
+ \lambda_2^*(t)(\mu_h + \gamma_h + \delta_h)\\[0.1 cm]
\dot{\lambda^*_3}(t) =  \lambda_3^*(t)(\lambda_v + \mu_{vb})
- \lambda_4^*(t)(\lambda_v) ) \\[0.1 cm]
\dot{\lambda^*_4}(t) = \lambda_4^*(t) \mu_{vb} \, ,
\end{cases}
\end{equation}
with transversality conditions
\begin{equation*}
\lambda^*_i(t_f) = 0,
\quad i=1, \ldots, 4 \, .
\end{equation*}
Furthermore,
\begin{equation}
\label{optcontrols}
u^*(t) = \min \left\{ \max \left\{0, \frac{\lambda_h(b) S_h^*(t)
}{C} \left(\lambda^*_2(t) - \lambda^*_1(t)\right)\right\}, 1 \right\} \, .
\end{equation}
\end{theorem}

\begin{proof}
Existence of an optimal solution $\left(S_h^*, I_h^*, S_v^*, I_v^*\right)$
associated to an optimal control $u^*$ comes from the convexity
of the integrand of the cost function $J$ with respect
to the control $u$ and the Lipschitz property of the state system
with respect to state variables $\left(S_h, I_h, S_v, I_v\right)$
(see, \textrm{e.g.}, \cite{Cesari_1983,Fleming_Rishel_1975}).
System \eqref{adjoint_function} is derived from the Pontryagin maximum principle
(see \eqref{adjsystemPMP}, \cite{Pontryagin_et_all_1962}) and the optimal controls
\eqref{optcontrols} come from the minimization condition \eqref{maxcondPMP}.
The optimal control pair given by \eqref{optcontrols} is unique due
to the boundedness of the state and adjoint functions and the Lipschitz property
of systems \eqref{model:malaria:controls} and \eqref{adjoint_function}
(see, e.g., \cite{SLenhart_2002} and references cited therein).
\end{proof}


\section*{Acknowledgements}

This work was supported by FEDER funds through COMPETE
--- Operational Programme Factors of Competitiveness
(``Programa Operacional Factores de Competitividade'')
and by Portuguese funds through the
Portuguese Foundation for Science and Technology
(``FCT --- Funda\c{c}\~{a}o para a Ci\^{e}ncia e a Tecnologia''),
within project PEst-C/MAT/UI4106/2011
with COMPETE number FCOMP-01-0124-FEDER-022690.
Silva was also supported by FCT through the
post-doc fellowship SFRH/BPD/72061/2010/J003420E03G;
Torres by EU funding under the 7th Framework Programme
FP7-PEOPLE-2010-ITN, grant agreement number 264735-SADCO.


\small



\end{document}